\documentclass[12pt]{amsart}

\usepackage{amscd,amssymb,amsopn,amsmath,amsthm,mathrsfs,graphics,amsfonts,enumerate,verbatim,calc
}

\usepackage[all]{xy}

\usepackage{color}

\usepackage[OT2,OT1]{fontenc}
\newcommand\cyr{%
\renewcommand\rmdefault{wncyr}%
\renewcommand\sfdefault{wncyss}%
\renewcommand\encodingdefault{OT2}%
\normalfont
\selectfont}
\DeclareTextFontCommand{\textcyr}{\cyr}

\usepackage{amssymb,amsmath}

\DeclareFontFamily{OT1}{rsfs}{}
\DeclareFontShape{OT1}{rsfs}{n}{it}{<-> rsfs10}{}
\DeclareMathAlphabet{\mathscr}{OT1}{rsfs}{n}{it}

\numberwithin{equation}{section}
\newtheorem{theorem}{Theorem}[section]
\newtheorem{lemma}[theorem]{Lemma}

\newtheorem{corollary}[theorem]{Corollary}

\newtheorem{question}{Question}
\newtheorem{notation}{Notation}

\newtheorem*{maintheorem}{Main Theorem}

\theoremstyle{definition}

\newtheorem{remark}[theorem]{Remark}
\theoremstyle{remark}

\newtheorem{acknowledgement}{Acknowledgement}

\begin{document}
\title[System of Parameters]{Uniform Annihilators of Systems of parameters}

\author[Pham Hung Quy]{Pham Hung Quy}
\address{Department of Mathematics, FPT University, Hanoi, Vietnam}
\email{quyph@fe.edu.vn}

\thanks{2010 {\em Mathematics Subject Classification\/}: 13H99, 13D45.\\
The work is partially supported by a fund of Vietnam National Foundation for Science
and Technology Development (NAFOSTED) under grant number 101.04-2020.10.}

\keywords{System of parameters, local cohomology}

\maketitle

\begin{abstract} In this paper we give uniform annihilators for some relations of all systems of parameters in a local ring. Our results provide simple proofs for some classical results.  
\end{abstract}

\section{Introduction}
In this note, let $(R, \frak m)$ be a Noetherian local ring. Suppose $\dim R = d$ and $\underline{x} = x_1, \ldots, x_d$ a system of parameters of $R$. For all $s \le d$, set $Q_s = (x_1, \ldots, x_s)$, and $Q = Q_d$. It is well known that if $R$ is Cohen-Macaulay then $x_1, \ldots, x_d$ is a regular sequence. In this case $x_1, \ldots, x_d$ act as independent variables. For all $0 \le s \le d-1$, any linear relation
$$a_1 x_1 + \cdots + a_s x_s + a_{s+1} x_{s+1} = 0$$
is trivial, i.e. $a_{s+1} \in (x_1, \ldots, x_s) : x_{s+1} = (x_1, \ldots, x_s)$. In general, $(x_1, \ldots, x_s) : x_{s+1}$ often strictly contains $(x_1, \ldots, x_s)$ and the situation becomes complicated. Schenzel \cite{Sch82} defined the ideal 
$$\frak b(R) = \bigcap_{\underline{x}, s <d} \mathrm{Ann}\big(\frac{Q_s : x_{s+1}}{Q_s} \big),$$
where $\underline{x}$ runs over all systems of parameters of $R$. The ideal $\frak b(R)$ controls all non-trivial linear relations of systems of parameters. By the definition we can not expect to determine $\frak b(R)$ explicitly, but {\it is $\frak b(R)$ a non-zero ideal}? For $i \le d$ let $\frak a_i(R) = \mathrm{Ann}(H^i_{\frak m}(R))$, and $\frak a(R) = \prod_{i=0}^{d-1} \frak a_i(R)$. Schenzel \cite[Satz 2.4.5]{Sch82} proved the following inclusions
$$\frak a(R) \subseteq \frak b(R) \subseteq \frak a_0(R) \cap \cdots \cap \frak a_{d-1}(R).$$ 
Therefore $\frak a(R)$ and $\frak b(R)$ have the same radical ideal. Importantly, if $R$ is a homomorphic image of a Cohen-Macaulay local ring then $\dim R/\frak a(R) < d$, so $\frak b(R) \neq 0$ in this case. Moreover we can choose a parameter element $x \in \frak b(R)$. This kind of parameter elements admits several nice properties, see \cite[Sections 3 and 4]{CQ16-2} for more details.\\
Another type of relation of system of parameters comes from the Hochster monomial conjecture. For all $s \le d$ we define the {\it limit closure} $Q_s^{\lim}$ of $Q_s$ as the formula
$$Q_s^{\lim}  = \bigcup_{n \ge 1} \big( (x_1^{n+1}, \ldots, x_s^{n+1}) : (x_1 \cdots x_s)^n \big).$$
Thanks to Hochster in equicharacteristic \cite{Hoc73} and Andr\'{e} in mixed characteristic \cite{Ad18} we have $Q^{\lim} \subseteq \frak m$.\footnote{Recently, Ma, Smirnov and the author \cite{MQS19} showed that $Q^{\lim} \subseteq \overline{Q}$, the integral closure of $Q$, for every parameter ideal $Q$ provided $R$ is quasi-unmixed.} One can check that (see \cite[Lemma 5.6]{PQ19}) 
$$\frak b(R)^s Q_s^{\lim} \subseteq Q_s$$
for all systems of parameters $\underline{x}$, and for all $s \le d$. In particular we have
$$\frak a(R)^d \subseteq \bigcap_{\underline{x}, s \le d}\mathrm{Ann}\big( \frac{Q_s^{\lim}}{Q_s} \big).$$
In this paper we are interested in relations concerning powers of parameter ideals. Inspired by $\frak b(R)$ it is natural to define the ideal  
$$\frak c(R) = \bigcap_{\underline{x}, s <d, n \ge 1} \mathrm{Ann}\big(\frac{Q_s^n : x_{s+1}}{Q_s^n} \big),$$
where $\underline{x}$ runs over all systems of parameters of $R$. It is clear that $\frak c(R) \subseteq \frak b(R)$.
\begin{question}\label{Q1}
Can we choose $N$, depending only on $d$, such that $\frak a(R)^N \subseteq \frak c(R)$?
\end{question}
We are also interested in the {\it parametric decomposition} of powers of parameter ideals. The parametric decomposition was considered in \cite[Section 3]{LT81}. Then it was studied by Heinzer, Ratliff and Shah in \cite{HRS95}, and followed by many researchers \cite{GS04-1,GS04-2,CT09}. For each $1 \le s \le d$ and $n \ge 1$ we set
$$\Lambda_{s, n} = \{(\alpha_1, \ldots, \alpha_s) \in \mathbb{N}^s \mid \alpha_i \ge 1 \,\text{for all}\, 1\le i \le s, \sum_{i=1}^s \alpha_i = s+n-1 \}.$$
For all $s \le d$ and all $\alpha \in \Lambda_{s, n}$, set $Q_s(\alpha) = (x_1^{\alpha_1}, \ldots, x_s^{\alpha_s})$. It is not hard to see that
$$Q_s^n  \subseteq \bigcap_{\alpha \in \Lambda_{s, n}} Q_s(\alpha)$$
for all $n \ge 1$. Moreover the equalities hold for all $s$ and $n$ provided $R$ is Cohen-Macaulay by \cite[Theorem 2.4]{HRS95} (see  Remark \ref{new proof HSR} for a short proof). In general, we consider the ideal
$$\frak d(R) = \bigcap_{\underline{x}, s \le d, n \ge 1} \mathrm{Ann} \big(\frac{\cap_{\alpha \in \Lambda_{s, n}} Q_s(\alpha)}{Q_s^n} \big),$$
where $\underline{x}$ runs over all systems of parameters of $R$. 
\begin{question} \label{Q2}
Can we choose $N$, depending only on $d$, such that $\frak a(R)^N \subseteq \frak d(R)$?
\end{question}
It is quite surprising that the two above questions are closely related. The aim of this paper is to give positive answers for both questions.
\begin{maintheorem}\label{maim thm} Let $(R, \frak m)$ be a local ring of dimension $d$. Then 
\begin{enumerate}
\item $\frak a(R) \subseteq \frak c(R)$ if $d = 1$, and $\frak a(R)^{2^{d-2}} \subseteq \frak c(R)$ if $d \ge 2$. 
\item $\frak a(R)^{2^{d-1}} \subseteq \frak d(R)$.
\end{enumerate}
\end{maintheorem}
\begin{remark} Recently, we learned that Question \ref{Q1} was studied by Lai when $R$ admits a dualizing complex by a very different method \cite[Theorem 3.2, Proposition 3.4]{L95}.
\end{remark}
\subsection{Some direct applications}
We present the usefulness of the Main Theorem in proving the colon capturing of integral closure and tight closure. The readers are encouraged  to compare our proof with the old ones \cite[Theorem 5.4.1]{HS06} and \cite[Theorem 10.1.9]{BH98}.
\begin{corollary} Let $(R, \frak m)$ be an equidimensional local ring of dimension $d$ that is a homomorphic image of a Cohen-Macaulay local ring. Let $\underline{x} = x_1, \ldots, x_d$ be a system of parameters of $R$. Then for all $s < d$ and all $n \ge 1$ we have 
\begin{enumerate}
\item $(x_1, \ldots, x_s)^n : x_{s+1} \subseteq \overline{(x_1, \ldots, x_s)^n} : x_{s+1} = \overline{(x_1, \ldots, x_s)^n},$ where $\overline{I}$ denotes the integral closure of $I$.\footnote{This assertion was stated for quasi-unmixed local ring. However we can reduce to our assumption by using \cite[Proposition 1.6.2]{HS06}.}
\item Suppose $R$ contains a field of prime characteristic $p>0$. Then we have
$$(x_1, \ldots, x_s)^n : x_{s+1} \subseteq ((x_1, \ldots, x_s)^n)^{*} : x_{s+1} = ((x_1, \ldots, x_s)^n)^{*},$$ where $I^{*}$ denotes the tight closure of $I$.
\end{enumerate}
\end{corollary}
\begin{proof} (1) The first inclusion is trivial. Let $r \in \overline{(x_1, \ldots, x_s)^n} : x_{s+1}$, we have $rx_{s+1} \in \overline{(x_1, \ldots, x_s)^n}$. By \cite[Corollary 6.8.12]{HS06} this condition is equivalent to the condition that we have an element $c \notin P$ for all $P \in \mathrm{minAss}(R)$ such that $c (rx_{s+1})^m \in Q_s^{nm}$ for large enough $m$. Therefore $cr^m  \in Q_s^{nm} : x_{s+1}^m$. Since $R$ is an image of a Cohen-Macaulay local ring, $\dim R/\frak a(R) < d$. By the Main Theorem we can choose $c' \in \frak c(R)$ such that $c' \notin P$ for all $P \in \mathrm{minAss}(R)$, reminding that $R$ is equidimensional. Therefore $(cc') r^m \in Q_s^{nm}$ for all $m$ large enough, and  so $r \in \overline{(x_1, \ldots, x_s)^n}$ by \cite[Corollary 6.8.12]{HS06}.\\
(2) We  recall that an element $x$  is contained in the tight closure $I^*$ of $I$ if we have some $c \notin P$ for all $P \in \mathrm{minAss}(R)$ such that $cx^{p^e} \in I^{[p^e]}$ for all $e \gg 0$, where $I^{[p^e]} = (a^{p^e} \mid a \in I)$ is the $e$-th Frobenius power of $I$, see \cite[Chapter 10]{BH98} for more details about theory of tight closure. We now can repeat the method of (1) to prove (2) with note that $(I^n)^{[p^e]} = (I^{[p^e]})^n$ for all ideal $I$ and all $n, e$.
\end{proof}  
We next prove that every power of a parameter ideal has a parametric decomposition up to tight closure. The result can be proved by using \cite[Section 7]{HH90}.
\begin{corollary}
  Let $(R, \frak m)$ be an equidimensional local ring of dimension $d$ and of prime characteristic $p$ that is a homomorphic image of a Cohen-Macaulay local ring. Let $\underline{x} = x_1, \ldots, x_d$ be a system of parameters of $R$. Then for all $s \le d$ and all $n \ge 1$ we have
 $$\bigcap_{\alpha \in \Lambda_{s,n}} Q_s(\alpha)^{*} = (Q_s^n)^{*}.$$
\end{corollary}
\begin{proof} Clearly, $\mathrm{RHS} \subseteq \mathrm{LHS}$. Let $r \in \cap_{\alpha \in \Lambda_{s,n}} Q_s(\alpha)^{*}$. Then we can choose $c \notin P$ for all $P \in \mathrm{minAss}(R)$ such that for all $q = p^e \gg 0$ we have
$$cr^q \in \bigcap_{\alpha \in \Lambda_{s,n}} (Q_s(\alpha))^{[q]} = \bigcap_{\alpha \in \Lambda_{s,n}} (Q_s^{[q]}(\alpha)).$$
By the Main Theorem (2) we have 
$$\frak a(R)^{2^{d-1}}\big( \bigcap_{\alpha \in \Lambda_{s,n}} (Q_s^{[q]}(\alpha)) \big) \subseteq (Q_s^{[q]})^n = (Q_s^n)^{[q]}$$
for all $q$. Therefore we can choose $c' \notin P$ for all $P \in \mathrm{minAss}(R)$ such that $(cc') r^q \in (Q_s^n)^{[q]}$ for all $q = p^e\gg 0$. Hence $r \in (Q_s^n)^{*}$. The proof is complete.
\end{proof}
At the time of writing we do not have an answer for the following question.
\begin{question}
Let $(R, \frak m)$ be an equidimensional local ring of dimension $d$ that is a homomorphic image of a Cohen-Macaulay local ring. Let $\underline{x} = x_1, \ldots, x_d$ be a system of parameters of $R$. Then is it true that
 $$\bigcap_{\alpha \in \Lambda_{s,n}} Q_s(\alpha) \subseteq \overline{Q_s^n}$$
 for all $s \le d$ and all $n \ge 1$?
\end{question}
The paper is organized as follows: We recall some notations of this paper in the next section. In Section 3 we prove a technical lemma which plays a key role in study the parametric decomposition. The Main Theorem will be proved in Section 4. We also give some applications of the Main Theorem about the uniform annihilator of local cohomology of $R/Q_s^n$.

\begin{acknowledgement} The author is deeply grateful to the referee for her/his careful reading and many value suggestions. This paper was written while the author visited Vietnam Institute for Advanced Study in Mathematics (VIASM), he would like to thank VIASM for the very kind support and hospitality.
\end{acknowledgement}
\section{Preliminaries}
For simplicity of the argument we will prove our results for the module version. In the rest of this paper, let $(R, \frak m)$ be a local ring of dimension $\dim R$ and $M$ a finitely generated $R$-module of dimension $d$. Let $\underline{x} = x_1, \ldots, x_d$ be a system of parameters of $M$. For $s \le d$ we set $Q_s = (x_1, \ldots, x_s)$ with the convention $Q_0  = (0)$.   
\begin{notation} \rm Let $(R, \frak m)$ be a local ring and $M$ a finitely generated $R$-module of dimension $d$.
\begin{enumerate}
\item For all $i < d$ we set  $\frak a_i(M) =
\mathrm{Ann}H^{i}_\mathfrak{m}(M)$, and $\frak a(M) = \frak
a_0(M) \cdots \frak a_{d-1}(M)$.
\item For all $s < d$ we set
 $$\mathfrak{b}_s(M) = \bigcap_{\underline{x}}
\mathrm{Ann}\big(\frac{Q_sM:x_{s+1}}{Q_sM} \big),$$ where $\underline{x} =
x_1, \ldots, x_d$ runs over all systems of parameters of $M$, and 
$$\frak b(M) = \frak b_0(M) \cap \cdots \cap \frak b_{d-1}(M).$$
\end{enumerate}
\end{notation}
\begin{remark}\label{C3.1.2}\rm
\begin{enumerate}
\item Schenzel \cite[Satz 2.4.5]{Sch82} proved that
$$\mathfrak{a}(M) \subseteq \mathfrak{b}(M) \subseteq \mathfrak{a}_0(M) \cap \cdots \cap \mathfrak{a}_{d-1}(M).$$
In particular $\sqrt{\mathfrak{a}(M)} = \sqrt{\mathfrak{b}(M)}$.
\item If $R$ is a homomorphic image of a Cohen-Macaulay local ring, then $\dim R/\mathfrak{a}_i(M) \leq
i$ for all $i< d$ by \cite[9.6.6]{BS98}.
\end{enumerate}
\end{remark}
We next mention the main objects of this paper.
\begin{notation} 
For all $s < d$ we set
$$\frak c_s(M) = \bigcap_{\underline{x}, n \ge 1} \mathrm{Ann}\big(\frac{Q_s^nM:x_{s+1}}{Q_s^nM} \big),$$
where $\underline{x} =
x_1, \ldots, x_d$ runs over all systems of parameters of $M$, and  
$$\frak c(M) = \frak c_0(M) \cap \cdots \cap \frak c_{d-1}(M).$$
\end{notation}
Clearly $\frak c_i(M) = \frak b_i(M)$ when $i = 0, 1$, and $\frak c_i(M) \subseteq \frak b_i(M)$ in general.\\ 
For the next object we need the notation
$$\Lambda_{s, n} = \{(\alpha_1, \ldots, \alpha_s) \in \mathbb{N}^s \mid \alpha_i \ge 1 \,\text{for all}\, 1\le i \le s, \sum_{i=1}^s \alpha_i = s+n-1 \}$$
for all $s, n \ge 1$. For all $s \le d$ and $\alpha \in \Lambda_{s, n}$, set $Q_s(\alpha) = (x_1^{\alpha_1}, \ldots, x_s^{\alpha_s})$. We have
$$Q_s^n M \subseteq \bigcap_{\alpha \in \Lambda_{s, n}} Q_s(\alpha)M$$
for all $n \ge 1$, and the equalities hold provided $M$ is Cohen-Macaulay \cite[Theorem 2.4]{HRS95} (see also \cite{GS04-1, GS04-2, CT09}). In this case we say that $Q_s^nM$ admit the {parametric decomposition} for all $n \ge 1$. In general, it is natural to consider the following.
\begin{notation}
For every $1 \le s \le d$, we set
$$\frak d_s(M) = \bigcap_{\underline{x}, n \ge 1} \mathrm{Ann} \big(\frac{\cap_{\alpha \in \Lambda_{s, n}} Q_s(\alpha)M}{Q_s^nM} \big),$$
where $\underline{x}$ runs over all systems of parameters of $R$, and
$$\frak d(M) = \frak d_1(M) \cap \cdots \cap \frak d_{d}(M).$$
\end{notation}
It is clear that $\frak d_1(M) = R$. The aim of this paper is to bound both $\frak c(M)$ and $\frak d(M)$ by $\frak b(M)$ (and $\frak a(M)$). We will need the following relation between these ideals.
\begin{lemma}\label{relations ideals} For every $1 \le s  \le d-1$ we have $ \frak b_s(M) \frak d_s(M)  \subseteq \frak c_s(M)$.
\end{lemma}
\begin{proof} For all $n \ge 1$ we have 
$$Q_s^nM : x_{s+1} \subseteq \bigcap_{\alpha \in \Lambda_{s, n}} \big(Q_s(\alpha)M : x_{s+1} \big).$$
By the definition of $\frak b_s(M)$ we have $\frak b_s(M)(Q_s(\alpha)M : x_{s+1}) \subseteq Q_s(\alpha)M$ for all $\alpha \in \Lambda_{s,n}$. Therefore
\begin{eqnarray*}
\frak b_s(M) \frak d_s(M) (Q_s^nM : x_{s+1}) &\subseteq &  \frak d_s(M) \frak b_s(M) \big( \bigcap_{\alpha \in \Lambda_{s, n}} (Q_s(\alpha)M : x_{s+1}) \big)\\
 &\subseteq &  \frak d_s(M) \big( \bigcap_{\alpha \in \Lambda_{s, n}} \frak b_s(M) (Q_s(\alpha)M : x_{s+1}) \big)\\
&\subseteq & \frak d_s(M) \big( \bigcap_{\alpha \in \Lambda_{s, n}} Q_s(\alpha)M\big)\\
&\subseteq & Q_s^nM.
\end{eqnarray*}
The proof is complete.
\end{proof}

\section{A technical Lemma}
Let $M$ be a finitely generated $R$-module, $I$ an ideal, and $x$ an element of $R$. We will use the notation $IM : x^{\infty} = \cup_{n \ge 1} (IM: x^n)$, and the conventions $I^0 = R$ and $x^0 = 1$.
 \begin{lemma}\label{intersection} Let $M$ be a finitely generated $R$-module, $I$ an ideal, and $x$ an element of $R$. Then
 \begin{enumerate}
 \item For all $n \ge m \ge 0$ we have $x^n M \cap x^m(IM: x^{\infty}) = x^n(IM : x^{\infty})$.
 \item For all $n + 1 \ge \alpha > m \ge 1$ we have
 $$x^mM \cap \big( \sum_{i = 0}^{m-2} x^i (I^{n-i}M: x^{\infty}) + x^{m-1}(I^{n+1 - \alpha}M : x^{\infty}) \big) = x^m(I^{n+1 - \alpha}M : x^{\infty}).$$
 \end{enumerate}
 
 \end{lemma}
\begin{proof}
(1) It is clear that $x^n(IM : x^{\infty}) \subseteq x^nM \cap x^m(IM : x^{\infty})$. Conversely, we have
$$x^nM \cap x^m(IM : x^{\infty}) \subseteq x^n M \cap (IM : x^{\infty}) = x^n ((IM : x^{\infty}):x^n) = x^n (IM : x^{\infty}).$$
(2) Similarly, it is enough to show $\mathrm{LHS} \subseteq \mathrm{RHS}$. Since $\alpha > m$ we have $n+1 - \alpha < n-i$ for all $i = 0, \ldots, m-2$. Therefore 
$$\sum_{i = 0}^{m-2} x^i (I^{n-i}M : x^{\infty}) + x^{m-1}(I^{n+1 - \alpha}M : x^{\infty})  \subseteq I^{n+1 - \alpha}M : x^{\infty}.$$
Hence
$$\mathrm{LHS} \subseteq x^mM \cap (I^{n+1 - \alpha}M : x^{\infty}) = x^m(I^{n+1 - \alpha}M : x^{\infty}).$$
The proof is complete.
\end{proof} 
The next lemma is the key ingredient of the proof of the main result. 
\begin{lemma} \label{key lemma} 
 Let $M$ be a finitely generated $R$-module, $I$ an ideal, and $x$ an element of $R$. Then for all $n \ge 1$ we have
 $$\bigcap_{\alpha =1}^{n} (x^{\alpha}M + I^{n+1 - \alpha}M : x^{\infty}) = \sum_{\alpha = 0}^n x^{\alpha}(I^{n - \alpha}M : x^{\infty}).$$
 \end{lemma}
\begin{proof}
We will prove by induction the following equalities 
$$\bigcap_{\alpha =1}^{n} (x^{\alpha}M + I^{n+1 - \alpha}M : x^{\infty}) = \bigcap_{\alpha = m}^{n} \big(x^{\alpha}M + \sum_{i = 0}^{m-2} x^i(I^{n-i}M : x^{\infty})+ x^{m-1}(I^{n+1 - \alpha}M : x^{\infty}) \big)$$
 for all $m = 1, \ldots, n$. Notice that if we apply $m=n$, we will obtain the desired. 
There is nothing to do in the case $m = 1$. Suppose we have the equality for some $m$, $1 \le m \le n-1$. Then\\

\noindent $\mathrm{LHS}$ 
\begin{eqnarray*}
&=& \bigcap_{\alpha = m+1}^{n} \bigg[\big(x^mM + \sum_{i = 0}^{m-1} x^i(I^{n-i}M : x^{\infty})\big) \cap \big(x^{\alpha}M + \sum_{i = 0}^{m-2} x^i(I^{n-i}M : x^{\infty}) + x^{m-1}(I^{n+1 - \alpha}M : x^{\infty}) \big)\bigg]\\
&=& \bigcap_{\alpha = m+1}^{n} \bigg[x^{\alpha}M + \sum_{i = 0}^{m-1} x^i(I^{n-i}M : x^{\infty}) + x^mM \cap  \big( \sum_{i = 0}^{m-2} x^i(I^{n-i}M : x^{\infty}) + x^{m-1}(I^{n+1 - \alpha}M : x^{\infty}) \big) \bigg]\\
&=& \bigcap_{\alpha = m+1}^{n} \big(x^{\alpha}M + \sum_{i = 0}^{m-1} x^i(I^{n-i}M : x^{\infty}) + x^m (I^{n+1 - \alpha}M : x^{\infty}) \big). 
\end{eqnarray*}
The last equality follows from Lemma \ref{intersection}(2). The proof is complete.
\end{proof} 
\begin{remark}\label{new proof HSR} The above Lemma gives us a direct proof for the fact 
$$Q_s^nM = \bigcap_{\alpha \in \Lambda_{s,n}} Q_s(\alpha)M$$ 
provided $x_1, \ldots, x_s$ is a regular sequence on $M$. Indeed, we can assume the ring is local and will proceed by induction on $s$. The case $s = 1$ is obvious. For $s >1$ we have
$$\bigcap_{\alpha \in \Lambda_{s,n}} Q_s(\alpha)M = \bigcap_{\alpha_s = 1}^n \big( \bigcap_{\alpha' \in \Lambda_{s-1,n+1 - \alpha_s}} (Q_{s-1}(\alpha')M + x_{s}^{\alpha_s}M) \big).$$
Applying the inductive hypothesis for the regular seuqence $x_1, \ldots,x_{s-1}$ on $M/x_s^{\alpha_s}M$ we have
$$\bigcap_{\alpha' \in \Lambda_{s-1,n+1 - \alpha_s}} Q_{s-1}(\alpha')(M/x_{s}^{\alpha_s}M) = Q_{s-1}^{n+1 - \alpha_s}(M/x_s^{\alpha_s}M).$$
Thus
$$\bigcap_{\alpha' \in \Lambda_{s-1,n+1 - \alpha_s}} (Q_{s-1}(\alpha')M+ x_{s}^{\alpha_s}M) = Q_{s-1}^{n+1 - \alpha_s}M +x_s^{\alpha_s}M$$
for all $\alpha_s = 1, \ldots, n$. Hence
$$\bigcap_{\alpha \in \Lambda_{s,n}} Q_s(\alpha)M = \bigcap_{\alpha_s = 1}^n (x_s^{\alpha_s}M + Q_{s-1}^{n+1 - \alpha_s}M).$$
 Now since $x_s$ is a regular element of $M/Q_{s-1}^mM$ for all $m \ge 1$ we have
 \begin{eqnarray*}
 \bigcap_{\alpha \in \Lambda_{s,n}} Q_s(\alpha)M &=& \bigcap_{\alpha_s = 1}^n (x_s^{\alpha_s}M + Q_{s-1}^{n+1 - \alpha_s}M: x_s^{\infty})\\
 &=& \sum_{\alpha = 0}^n x_s^{\alpha_s}(Q_{s-1}^{n - \alpha_s}M:x_s^{\infty}) \quad (\text{By Lemma}\,\, \ref{key lemma})\\
 &=&\sum_{\alpha = 0}^n x_s^{\alpha_s}Q_{s-1}^{n - \alpha_s}M\\
 &=& Q_s^nM.
 \end{eqnarray*}
 \end{remark}
\section{Proof of the main result}
We start with the following.
\begin{remark}\label{quotient} Let $(R, \frak m)$ be a local ring and $M$ a finitely generated $R$-module of dimension $d>0$. Let $x_1$ be a parameter element of $M$. Then we have
$$\frak b(M/x_1M) = \bigcap_{\underline{y}, i \le d-1} \mathrm{Ann}\big(\frac{(x_1, y_2, \ldots, y_i)M : y_{i+1}}{(x_1, y_2, \ldots, y_i)M} \big),$$
where $\underline{y} = y_2, \ldots, y_d$ runs over all systems of parameters of $M/x_1M$. Moreover, $\frak b(M) \subseteq \frak b(M/x_1M)$ since for every system of parameters $y_2, \ldots, y_d$ of $M/x_1M$ we have $x_1, y_2, \ldots, y_d$ is a system of parameters of $R$. More general, for every part of system of parameters $x_1,\ldots,x_i, i<d,$ of $M$ we have $\frak b(M) \subseteq \frak b(M/(x_1, \ldots,x_i)M)$.  
\end{remark}
We prove the main result of this section which immediately implies the Main Theorem introduced in the introduction.
\begin{theorem}\label{MT} Let $(R, \frak m)$ be a local ring and $M$ a finitely generated $R$-module of dimension $d>0$. Then
\begin{enumerate}
\item We have $\frak b(M)^{2^{s-1}} \subseteq \frak c_s(M)$ for all $s = 1, \ldots, d-1$, and $\frak b(M) \subseteq \frak c_0(M)$.
\item We have $\frak b(M)^{2^{s-1}-1} \subseteq \frak d_s(M)$ for all $s = 1, \ldots, d$.
\end{enumerate}
In particular for all systems of parameters $x_1, \ldots,x_d$ of $M$, and for all $n \ge 1$ we have
$$\frak a(M)^{2^{d-2}}\big((x_1, \ldots,x_s)^nM:x_{s+1}\big) \subseteq (x_1, \ldots,x_s)^nM$$
for all $1 \le s \le d-1$, and 
$$\frak a(M)^{2^{d-1}}\big(\bigcap_{\alpha \in \Lambda_{s,n}}Q_s(\alpha)M\big)\subseteq (x_1, \ldots,x_s)^nM$$
for all $s \le d$.
\end{theorem}
\begin{proof} By the definition we have $\frak b(M) \subseteq \frak c_0(M) \cap \frak c_1(M)$ and $\frak d_1(M) = R$. We will prove both (1) and (2) by induction on $s$. The case $s = 1$ was done. By Lemma \ref{relations ideals} we need only to prove that if $\frak b(M)^{2^{s-1}-1} \subseteq \frak d_s(M)$ and $\frak b(M)^{2^{s-1}} \subseteq \frak c_s(M)$ for some $s < d$, then $\frak b(M)^{2^{s}-1} \subseteq \frak d_{s+1}(M)$. For all systems of parameters $x_1, \ldots, x_d$ and all $n \ge 1$ we have  
$$\frak b(M)^{2^{s-1}-1} \big(\bigcap_{\alpha \in \Lambda_{s+1,n}} Q_{s+1}(\alpha)M\big) \subseteq \bigcap_{\alpha_{s+1} = 1}^n  \frak b(M)^{2^{s-1}-1}\bigcap_{\alpha' \in \Lambda_{s,n+1 - \alpha_{s+1}}} (Q_{s}(\alpha')M + x_{s+1}^{\alpha_{s+1}}M) \big).$$
By Remark \ref{quotient} we have $\frak b(M) \subseteq \frak b(M/x_{s+1}^{\alpha_{s+1}}M)$ for all $1 \le \alpha_{s+1} \le n$, with 
$$\frak b(M/x_{s+1}^{\alpha_{s+1}}M) = \bigcap_{\underline{y}, i \le d-2} \mathrm{Ann}\big(\frac{(x_{s+1}^{\alpha_{s+1}}, y_1, \ldots, y_i)M : y_{i+1}}{(x_{s+1}^{\alpha_{s+1}}, y_1, \ldots, y_i)M} \big).$$
Similarly we note that
$$\frak d_s(M/x_{s+1}^{\alpha_{s+1}}M) = \bigcap_{\underline{y}, n \ge 1} \mathrm{Ann} \big(\frac{\cap_{\alpha \in \Lambda_{s, n}} Q_s(\alpha)M/x_{s+1}^{\alpha_{s+1}}M}{Q_s^n \, M/x_{s+1}^{\alpha_{s+1}}M} \big),$$
where $\underline{y} = y_1, \ldots, y_{d-1}$ runs over all system of parameters of $M/x_{s+1}^{\alpha_{s+1}}M$. By using the induction for $M/x_{s+1}^{\alpha_{s+1}}M, 1 \le \alpha_{s+1} \le n$ we have 
$$ \frak b(M)^{2^{s-1}-1} \subseteq \frak b(M/x_{s+1}^{\alpha_{s+1}}M)^{2^{s-1}-1} \subseteq \frak d_{s} (M/x_{s+1}^{\alpha_{s+1}}M),$$
the least inequality is just $\frak b(M)^{2^{s-1}-1} \subseteq \frak d_s(M)$ but we apply for $M/x_{s+1}^{\alpha_{s+1}}M$. 
 Therefore 
$$\frak b(M)^{2^{s-1}-1}\big( \bigcap_{\alpha \in \Lambda_{s+1,n}} Q_{s+1}(\alpha)M \big) \subseteq \bigcap_{\alpha_{s+1} = 1}^n (Q_{s}^{n+1-\alpha_{s+1}}M + x_{s+1}^{\alpha_{s+1}}M).$$
Hence
\begin{eqnarray*}
\frak b(M)^{2^{s}-1}\big( \bigcap_{\alpha \in \Lambda_{s+1,n}} Q_{s+1}(\alpha)M \big) &\subseteq& \frak b(M)^{2^{s-1}}\big( \bigcap_{\alpha_{s+1} = 1}^n (Q_{s}^{n+1-\alpha_{s+1}}M + x_{s+1}^{\alpha_{s+1}}M) \big)\\
&\subseteq& \frak b(M)^{2^{s-1}} \big(\bigcap_{\alpha_{s+1} = 1}^n (x_{s+1}^{\alpha_{s+1}}M + Q_{s}^{n+1-\alpha_{s+1}}M: x_{s+1}^{\infty}))\\
&=& \frak b(M)^{2^{s-1}} \sum_{\alpha_{s+1} = 0}^n x_{s+1}^{\alpha_{s+1}}(Q_{s}^{n-\alpha_{s+1}}M: x_{s+1}^{\infty}) \quad (\text{By Lemma}\, \ref{key lemma})\\
&\subseteq& \sum_{\alpha_{s+1} = 0}^n x_{s+1}^{\alpha_{s+1}}Q_{s}^{n-\alpha_{s+1}}M \quad  \quad \quad (\text{Since }\,\frak b(M)^{2^{s-1}} \subseteq \frak c_s(M) )\\
&=& Q_{s+1}^nM.
\end{eqnarray*}
We finish the inductive proof. The last claims are easy since $\frak a(M) \subseteq \frak b(M)$. Hence the proof is complete.
\end{proof}
We next give some applications of Theorem \ref{MT} to annihilator of local cohomology of quotient rings $M/Q_s^nM$.
\begin{corollary} Let $(R, \frak m)$ be a local ring and $M$ a finitely generated $R$-module of dimension $d>0$. Let $\underline{x} = x_1, \ldots, x_d$ be a system of parameters of $M$. Then for every $1 \le s \le d-1$ we have $\frak b(M)^{2^{s-1}} \subseteq \frak b(M/Q_s^nM)$ for all $n \ge 1$.
\end{corollary}
\begin{proof} By the definition we have
$$\frak b(M/Q_s^nM) = \bigcap_{\underline{y}, i< d-s} \mathrm{Ann} \big(\frac{(Q_s^n, y_1, \ldots, y_i)M: y_{i+1}}{(Q_s^n, y_1, \ldots, y_i)M}\big),$$
where $\underline{y} = y_1, \ldots, y_{d-s}$ runs over all systems of parameters of $M/Q_s^nM$. Applying Theorem \ref{MT} for $M/(y_1, \ldots, y_i)M$ we have 
$$\frak b(M/(y_1, \ldots, y_i)M)^{2^{s-1}} \subseteq \frak c_s(M/(y_1, \ldots, y_i)M).$$
Therefore 
$$\frak b(M/(y_1, \ldots, y_i)M)^{2^{s-1}} \big(\frac{(Q_s^n, y_1, \ldots, y_i)M: y_{i+1}}{(Q_s^n, y_1, \ldots, y_i)M}\big) = 0.$$
On the other hand $\frak b(M) \subseteq \frak b(M/(y_1, \ldots, y_i)M)$ by Remark \ref{quotient}. Hence 
$$\frak b(M)^{2^{s-1}} ((Q_s^n, y_1, \ldots, y_i)M: y_{i+1}) \subseteq (Q_s^n, y_1, \ldots, y_i)M$$
for all $\underline{y}$ and all $i < d-s$. The proof is complete.
\end{proof}

\begin{corollary}
Let $(R, \frak m)$ be a local ring and $M$ a finitely generated $R$-module of dimension $d>0$. Then for all $1 \le s \le d-1$ and for all systems of parameters $\underline{x} = x_1, \ldots, x_d$ of $M$ we have
$$\frak a(M)^{2^{s-1}} H^i_{\frak m} (M/Q_s^nM) = 0$$
for all $n\ge 1$ and for all $i < d-s$.
\end{corollary}

\begin{corollary}
Let $(R, \frak m)$ be a local ring and $M$ a $k$-Buchsbaum $R$-module of dimension $d>0$ i.e. $\frak m^k H^i_{\frak m}(M) = 0$ for all $i<d$. Then for all $1 \le s \le d-1$ and for all systems of parameters $\underline{x} = x_1, \ldots, x_d$ of $M$ we have
$$\frak m^{kd2^{s-1}} H^i_{\frak m} (M/Q_s^nM) = 0$$
for all $n\ge 1$ and for all $i < d-s$.
\end{corollary}
\begin{proof} If $M$ is $k$-Buchsbaum then we have $\frak m^{kd} \subseteq \frak a (R)$. Therefore the assertion is a special case of the previous  result.
\end{proof}
We finally note that $\ell(H^i_{\frak m} (M/Q_s^nM))$ maybe a polynomial ring in $n$ of degree $s-1$ by \cite[Theorem 3.10]{CNQ18}.

\end{document}